\documentclass[12pt]{amsart}

\usepackage{amssymb}
\usepackage[bookmarks]{hyperref}
\usepackage[utf8]{inputenc}

\usepackage{enumerate}

\usepackage{graphicx}

\makeatletter
\@namedef{subjclassname@2010}{%
  \textup{2010} Mathematics Subject Classification}
\makeatother

\newtheorem{thm}{Theorem}[section]
\newtheorem{cor}[thm]{Corollary}

\newtheorem{prop}[thm]{Proposition}

\newtheorem{mainthm}[thm]{Main Theorem}

\theoremstyle{definition}
\newtheorem{defin}[thm]{Definition}

\newtheorem*{xrem}{Remark}

\numberwithin{equation}{section}

\begin{document}

%%%%% To ease editing, for IMPAN journals add:

\baselineskip=17pt

%%%%%%%%%%%

%% In the running head, replace first names by initials 
%% and give an abbreviation of the title.

\title{Tameness in the Fréchet spaces of analytic functions}

%    Information for first author
\author{Aydın Aytuna}
%    Address of record for the research reported here
\address{MDBF, Sabancı Üniversitesi. Orhanlı, Tuzla 34956 İstanbul}
\email{aytuna@sabanciuniv.edu}
%    \thanks will become a 1st page footnote.
%\thanks{It is my pleasure to thank Tosun Terzioğlu for fruitful discussions.}

%    General info
%\date{January 1, 2001 and, in revised form, June 22, 2001.}
\dedicatory{Dedicated to the memory of Tosun Terzioğlu}

\begin{abstract}
A Fréchet space, $\mathcal{X}$, with a sequence of generating semi-norms $\{\|\cdot\|_k\}_{k=1}^\infty$ is called \emph{tame} in case there exists an increasing function $\sigma: \ \mathbb{N} \rightarrow \mathbb{N}$, such that for every continuous linear operator $T$ from $\mathcal{X}$ into itself, there exists an $N_0$ and $C>0$ such that 
\[
    \|T(x)\|_n \leq C\|x\|_{\sigma(n)}, \quad \forall x \in \mathcal{X} \text{ and } n \geq N_0.
\]
This property does not depend upon the choice of fundamental system of semi-norms for $\mathcal{X}$ and is a property of the Fréchet space $\mathcal{X}$. In this paper we investigate tameness in the Fréchet spaces $\mathcal{O}(M)$ of analytic functions on Stein manifolds $M$ equipped with the compact open topology. Actually we will look into tameness in the more general class of nuclear Fréchet spaces with the properties $\underline{DN}$ and $\Omega$ of Vogt and then specialize to analytic function spaces. We will show that for a Stein manifold $M$, tameness of $\mathcal{O}(M)$ is equivalent to the hyperconvexity of $M$.
\end{abstract}

\maketitle

\section{Introduction}
Tameness of Fréchet spaces is an important property frequently used in functional analysis since it brings a kind of control on the otherwise chaotic behaviour of continuous linear endomorphisms. This concept was used effectively in the structure theory of nuclear Fréchet spaces, especially in finding bases in complemented subspaces of certain infinite type power series spaces \cite{DF}. Frequently the Fréchet spaces that arise in practical applications (e.g. in non-linear analysis), enjoy (versions of) the tameness property (\cite{HAMILTON, Z}). In fact, an inquiry about the tameness of the analytic function spaces was conveyed, sometime ago, in this context by D. Zarnadze [private communication]. In this paper we answer this question in a kind of negative way. We show that there are no tame analytic function spaces other than the natural ones (Theorem \ref{theorem:3}).

The organization of the paper is as follows: After establishing the notation and terminology, in section 1, we recall the definition of the linear topological invariants $\underline{DN}$ and $\Omega$ and introduce the technical tool that we will use in the later sections, namely local imbeddings of power series spaces of finite type into Fréchet spaces. After some general results on local imbeddings we establish a link between the approximate diametral dimension of the space and the existence of effective local imbeddings from certain finite type power series spaces (Theorem \ref{theorem:1}).

Section 2 is devoted to Theorem \ref{theorem:2} and its proof which characterizes tame nuclear Fréchet spaces having a finitely nuclear stable exponent sequence and enjoying the properties $\underline{DN}$ and $\Omega$.

In the first subsection of section 3 we show how the local imbeddings can be used to construct Green's functions on complex manifolds. The second subsection of the same section is devoted to diametral dimension considerations of analytic function spaces and the proof of the main theorem Theorem \ref{theorem:3} of this paper.

\subsection*{Notations and Terminology:}
We will use the terminology of \cite{M-V} and refer the reader to this book for the undefined concepts and the standard results of functional analysis that we will use. For the notions from complex potential theory used (especially in section 3) we refer the reader to \cite{KLIMEK}.

Power series sequence spaces play an important role in this paper. Recall that these are Fréchet spaces 
\[
    \Lambda_R(\alpha) \doteq \{(\xi_n)^\infty_{n=0}: \  |(\xi_n)|_r \doteq \left( \sum_n |\xi_n|^2 e^{2r\alpha_n} \right)^\frac{1}{2} < \infty, \quad \forall\ -\infty<r<R\}
\]
where $R$ is either $1$ or $\infty$, and $\alpha = (\alpha_n)_n$ is a sequence of complex numbers with $\sup \frac{\ln (n)}{\alpha_n} < +\infty$, called the \emph{exponent sequence} of the space. The grading on these spaces will be the hilbertian grading $\{|\cdot|_r\}_{r<R}$.

These spaces are referred to as \emph{finite type power} series spaces if $R=1$ and \emph{infinite type power} series spaces if $R=\infty$, respectively. We will use the notation $\Lambda_1[\alpha]$ to denote the Hilbert space $\{(\xi_n)^\infty_{n=0}: \ |(\xi_n)|_r \doteq \left( \sum_n |\xi_n|^2 e^{2r\alpha_n}\right)^\frac{1}{2},$ $\forall \ 0<r<1\}$ with the norm $|\cdot|_r$, $r \in \mathbb{R}$.

For a pair of Fréchet spaces $\{\mathcal{X}, \ ^{\|\cdot\|_k}\}_k$ and $\{\mathcal{Y}, \ ^{|\ |_k}\}_k$ the set of all continuous linear operators will be denoted by $\mathcal{L}(\mathcal{X}, \mathcal{Y})$. We will use the symbol
\[
    \|T\|^n_m \doteq \sup_{\|x\|_{n\leq 1}} |T x|_m, \quad n,m \in \mathbb{N}
\]
for a $T \in \mathcal{L}(\mathcal{X}, \mathcal{Y})$. Note that $\|\cdot\|_m^n$ can take the value $+\infty$.

For a Fréchet space, $\{\mathcal{X}, \ ^{\|\cdot\|_k}\}_k$, the local Hilbert spaces corresponding to the norm $\|\ \|_k$ will be denoted by $\mathcal{X}_{k}$, $k=1,2,\dots$. The closed unit ball in $\mathcal{X}$ corresponding to the semi-norm $\|\ \|_k$ will be denoted by $U_k$, $k=1,2,\dots$.

Throughout the paper we will reserve the symbol $\epsilon_n$ to denote the sequence $(0,0,\dots,0,1,0,\dots)$ where $1$ is in the $n^{\text{th}}$ place.

\section{Local Imbeddings of Power Series Spaces of Finite Type into Fréchet Spaces}
In this section we will examine local imbeddings of power series spaces of finite type into nuclear Fréchet spaces with the properties $\underline{DN}$ and $\Omega$.

In subsection \ref{subs:1.1} we recall the definition of the linear topological invariants $\underline{DN}$ and $\Omega$ of Vogt and list some properties of Fréchet spaces with properties $\underline{DN}$ and $\Omega$, that will be used in the sequel. We emphasize again that all the Fréchet spaces that we will deal with in this paper will be assumed to be nuclear and satisfy the properties $\underline{DN}$ and $\Omega$ unless otherwise specified. In subsection \ref{subs:1.2} we will investigate the existence of local imbeddings. In subsection \ref{subs:1.3} we consider Fréchet spaces whose diametral approximative dimension is equal to a finite type power series space and show that this property yields local imbeddings.

\begin{defin} \label{def:1}
A continuous linear operator $T$ from $\Lambda_1(\alpha)$ into a graded Fréchet space $\{\mathcal{X} , \ ^{{\|\cdot\|}_n }\}$ is called an \emph{$(r , k)$-local imbedding} in case;
\[
\exists \  C > 0 : \quad {\|T(x)\|}_k \geq C {|x|}_r , \quad \forall \  x \in \Lambda_1 (\alpha).
\]
We will say that a Fréchet space $\mathcal{X}$ admits an \emph{r-local imbedding} from $\Lambda_1 (\alpha)$, in case there exists a continuous linear operator $T$ from $\Lambda_1 (\alpha)$ in $\mathcal{X}$ with the property;
\begin{align*}
\exists \text{ continuous semi-norm } \ ^{\|\cdot\|} \text{ of } \mathcal{X} \text{ and } C > 0 \\ 
\text{ such that } \|T(x)\| \geq C |x|_r \quad \forall \ x \in \Lambda_1 (\alpha).
\end{align*}
\end{defin}

\subsection{Linear topological invariants $\underline{DN}$, $\Omega$ and associated exponent sequences} \label{subs:1.1}

\begin{defin} \label{def:2}
A nuclear Fréchet space $\mathcal{X}$ is said to have properties $\underline{DN}$ and $\Omega$ (\cite{Vogt1}, \cite{Vogt-Wagner1}) in case there exists a fundamental system of hilbertian norms $\{\|\ \|_k\}_k$ generating the topology of $\mathcal{X}$ which satisfy
\begin{align*}
(\underline{DN}) : \qquad & \forall k \ \exists \ 0 < \lambda < 1 \text{ and } C>0 \text{ such that } \\
& \quad {\|x\|}_{k+1} \leq C \ {\|x\|}^\lambda_k \  {\|x\|}^{1-\lambda}_{k+2} \quad \forall \ x \in \mathcal{X}  \\
(\Omega) : \qquad & \forall k \ \forall p \  \exists C>0 \text{, }  j \in \mathbb{N} \text{ such that } \\
&\quad U_{k+1} \subset C r^j U_p + \frac{1}{r} U_k \quad \forall \ r>0
\end{align*}
\end{defin}

For more information about these invariants and examples of Fréchet spaces possessing these properties we refer the reader to \cite{M-V}.

Let $U$ and $V$, $V \subseteq U$ be 2 subsets of a Fréchet space $\mathcal{X}$. We will use the notation $\mathcal{E}_n (U,V)$ to denote
\[
    \mathcal{E}_n (U,V) \doteq -\ln d_n (V, U),\quad n=0,1,2
\]
where
\[
    d_n (V,U) \doteq \inf_{L_n} \  \inf \left\{ \lambda : \  \lambda > 0, \ V \subseteq \lambda U + L_n\right\}
\]
(infimum is taken over all n-dimensional subspaces $L_n$ of $\mathcal{X}$) is the $n^{th}$ Kolmogorov diameter of $V$ in $U$.

In the case of a graded Fréchet space $(\mathcal{X} , \ ^{{\|\ \|}_k})$, we simplify this notation by setting 
\[
    \mathcal{E}_n (p, k) \doteq \mathcal{E}_n (U_p, U_k), \quad \forall \  p<k.
\]

Let us now fix a nuclear Fréchet space $\mathcal{X}$, with the properties $\underline{DN}$ and $\Omega$, and choose the grading ${\left\{ {\|\ \|}_k \right\}}_{k}$, derived from these properties. In \cite{AKT2} (c.f \cite{AKT1}) we have shown that all the sequences ${\{\mathcal{E}_n(k, k+1)\}}_n$, $k = 0,1,2,\dots$ are equivalent and called this equivalence class the \emph{exponent sequence associated} to $\mathcal{X}$. In this paper we, unless stated otherwise, will use a concrete representation, $\mathcal{E}_n \doteq \mathcal{E}_n (0,1), \ n=0,1,2$, of the associate exponent sequence of $\mathcal{X}$.

We compile (\cite{Terzioglu1}) some properties of these sequences that will be used in the sequel.

\begin{prop} \label{prop:1.1}
{\em Let $(\mathcal{X} , \ ^{{\|\ \|}_k})$ be as above. An $\Omega$-type condition;
\[
    \exists \  0<\lambda<1, D>0 : \quad U_q \subset r^{\frac{1-\lambda}{\lambda}} U_p + \frac{D}{r} U_k , \quad \forall \  r>C \quad (p<q<1)
\]
implies the inequalities:
\begin{enumerate}
\item $\exists \  C>0$; $\mathcal{E}_n (q, k) \leq \lambda \mathcal{E}_n (p;k)+C$,  $\forall \ n$,
\item $\exists \  C>0$; $\mathcal{E}_n (p, k) \leq \frac{1}{1-\lambda} \mathcal{E}_n (p,q)+C$,  $\forall \ n$, 
\item $\exists \  C>0$; $\mathcal{E}_n (q, k) \leq \frac{\lambda}{1-\lambda} \mathcal{E}_n (p,q)+C$,  $\forall \ n$,
%\item 
\end{enumerate}

A $\underline{DN}$-type condition:
\[
    \exists \  0<\lambda<1, D>0 : \quad {\|\cdot\|}_q \leq D {\|\cdot\|}_k^{1-\lambda} {\|\cdot\|}_p^\lambda \quad (p<q<k)
\]
implies the inequalities:
\begin{enumerate}
\item [(1')] $\exists \  C>0$; $\mathcal{E}_n (p, k) \leq \frac{1}{\lambda} \mathcal{E}_n (q;k)+C$,  $n=0,1,\dots$,
\item [(2')] $\exists \  C>0$; $\mathcal{E}_n (p, q) \leq (1-\lambda) \mathcal{E}_n (p;k)+C$,  $n=0,1,\dots$,
\item [(3')] $\exists \  C>0$; $\mathcal{E}_n (p, q) \leq \frac{\lambda}{1-\lambda} \mathcal{E}_n (q,k)+C$,  $n=0,1,\dots$,
%\item [(4')] 
\end{enumerate}
In the presence of both $\Omega$ and $DN$-type conditions we have
\begin{enumerate}
\item [(4)] $\exists \  C>0, \ 0<\lambda_1 <1,\ 0<\lambda_2 <1;\ \mathcal{E}_n \leq \frac{\lambda_1 \lambda_2}{(1-\lambda_1)(1-\lambda_2)} \mathcal{E}_n (p,q)+C,\ \forall \ n$
\item [(4')] $\exists \  C>0, \ 0<\lambda_1 <1,\ 0<\lambda_2 <1;\ \mathcal{E}_n (p,q) \leq \frac{\lambda_1 \lambda_2}{(1-\lambda_1)(1-\lambda_2)} \mathcal{E}_n + C,\  n=0,1,...$
\end{enumerate}
}
\end{prop}

\begin{proof}
See \cite{Terzioglu1} (c.f \cite{AKT2}).
\end{proof}

\subsection{Local imbeddings of power series spaces of finite type into nuclear Fréchet spaces with properties $\underline{DN}$ and $\Omega$} \label{subs:1.2}

\ 

We now turn our attention to the question of existence of local imbeddings into the graded Fréchet space $(\mathcal{X} , \ ^{{\|\ \|}_k})$ fixed in subsection \ref{subs:1.1} above. But we first consider some general considerations concerning local imbeddings.

\begin{prop} \label{prop:1.2}
{\em Let $\Lambda_1 (\alpha)$ be a nuclear power series space of finite type. The following assertions are equivalent;
\begin{enumerate}
\item [(i)] There exists an $(r,k)$-local imbedding from $\Lambda_1 (\alpha)$ into $\mathcal{X}$
\item [(ii)] There exists a sequence $\left\{ g_n \right\}_{n=0}^\infty$ in $\mathcal{X}$ that is orthogonal in $\mathcal{X}_k$ and satisfies; 
\[
    \overline{\lim}_n \frac{\ln \|g_n\|_s}{\alpha_n} < 1-r, \quad \forall \  s 
\]
\item [(iii)] There exists an isometry from the local Hilbert space $\Lambda_r [\alpha]$ into $H_k$ that induces a continuous linear operator from $\Lambda_1 (\alpha)$ into $\mathcal{X}$
\item [(iv)] There exists a closed, Hilbertian bounded disc $B$ which satisfies;
\[
    \exists \  C>0 \text{ such that } d_n (B, U_k) \geq C e^{-(1-r)\alpha_n}
\]
\end{enumerate}
}
\end{prop}

\begin{proof}
\ 
\begin{enumerate}

\item [(i)$\Rightarrow$(ii):] Fix an $(r,k)$-local imbedding $T$ with $\|T(x)\|_k \geq C_1 |x|_r$, $\forall \ x \in \mathcal{X}$. Consider 
\[
    f_n \doteq \frac{\epsilon_n}{e^{r\alpha_n}},\quad n=0,1,2,\dots
\]
and apply the Gram-Schmidt orthogonalization procedure to the finitely linearly independent sequence $\left\{ T(f_n)\right\}_n$ in $\mathcal{X}_k$ to get a sequence,
\[
    g_n \doteq \sum_{i=0}^n c_i^n T(f_i), \quad n = 1,2,\dots 
\]
in $\mathcal{X}$ that is orthogonal in $\mathcal{X}_k$. We estimate;
\[
    1 = \|g_n\|_{k} = \left\|T\left(\sum_{i=0}^{n} c_i^n g_i \right)\right\|_{k} \geq C_1 \left| \sum_{i=0}^{n} e_i^n \frac{\epsilon_i}{e^{r\alpha_i}} \right|_{r} = C_1 \left( \sum |c_i^n|^2\right)^{\frac{1}{2}}.
\]
From continuity of $T$ we get for every $s$, a $r < \sigma(s) < 1$ and $C>0$ such that, 
\begin{align*}
    \|g_n\|_s & = \left\| T\left( \sum_{i=0}^{n} c_i^n f_i \right) \right\|_s \\ & \leq C \left| \sum_{i=0}^{n} c_i^n f_i \right|_{\sigma(s)} \\ & \leq C e^{(\sigma(s)-r)_{\alpha n}}\left( \sum_{i=0}^n |c_i|^2 \right)^{\frac{1}{2}} \leq \frac{C}{C_1} e^{(\sigma(s)-r)_{\alpha n}}.
\end{align*}
Hence
\[
    \overline{\lim}_n \frac{\ln \|g_n\|_s}{\alpha_n} \leq \sigma(s)-r < 1-r.
\]

\item [(ii)$\Rightarrow$(iii):] Choose $\{g_n\}_n$ as in (ii) and set 
\[
    T\left(\sum_{i=0}^{\infty} c_i \epsilon_i \right) \doteq \sum_{i=0}^{\infty} e^{r\alpha_i} c_i g_i.
\]
This operator is plainly a continuous linear operator from $\Lambda_1(\alpha)$ into $\mathcal{X}$ that extends to an isometry from $\Lambda_r[\alpha_{n}]$ into $\mathcal{X}_k$.

\item [(iii)$\Rightarrow$(iv):] Choose an operator $T$ which has the properties given in (iii) and set;
\[
    g_n \doteq T \left( \frac{\epsilon_n}{e^{r\alpha_n}} \right), \quad n = 0,1,\dots.
\]
Consider the set
\[
    B \doteq \left\{ x \in \mathcal{X}; \quad \sum_{n=0}^\infty |\langle x, g_n \rangle_k|^2 e^{2(1-r)\alpha_n} \leq 1 \right\},
\]
where $\langle , \rangle_k$ is the inner product corresponding to $\|\ \|_k$.

A direct computation will show that $B$ is the image under $T$ of the compact set $B_0=\{(\xi_n); \quad \sum |\xi_n|^2 e^{2\alpha_n} \leq 1\}$ in $\Lambda_1(\alpha)$. So $B$ is a compact disc that is Hilbertian and it follows that (see Lemma 6.2.2 of \cite{D})
\[
    d_n (B,U_k) \geq d_n(B \cap T(\mathcal{X}), T(\mathcal{X}) \cap U_k) = e^{(r-1)\alpha_n}
\]

\item [(iv)$\Rightarrow$(i):] Let $H_B$ denote the Hilbert space generated by $B$ in $\mathcal{X}$. Since the inclusion $H_B \hookrightarrow \mathcal{X}_k$ is a compact operator we can choose a sequence $\{f_n\}_{n=0}^\infty$ of orthogonal vectors in $H_B$ that are orthogonal in the Hilbert space $\mathcal{X}_k$ and $\|f_n\|_B = \left( d_n(B, U_k) \right)^{-1}, \quad n=0,1,2,\dots$. Set
\[
    T\big((\xi_n)_n\big) \doteq \sum_{n=0}^\infty \xi_n e^{r\alpha_n} f_n, \quad \forall \ (\xi_n) \in \Lambda_1(\alpha).
\]
We fix an $s>k$ and choose, in view of the condition $\underline{DN}$, a $0<\lambda<1$ and $C>0$ such that $\|x\|_s \leq C \|x\|_k^{1-\lambda} \|x\|_B^\lambda, \quad  \forall\ x \in H_B$.

Choose $\lambda^+$ so that, $0<r+\lambda(1-\lambda)<\lambda^+<1$. There exists positive constants $C_1, C_2$ such that;
\begin{align*}
\left\| T\big((\xi_n)\big)_n \right\|_s  & \leq \sum_n |\xi_n| e^{r\alpha_n} \|f_n\|_s \\ & \leq C_1 \sum_n |\xi_n| e^{r\alpha_n} e^{\lambda(1-r)\alpha_n}\\ & \leq C_2\left( \sum_n |\xi_n|^2 e^{2\lambda^+ \alpha_n} \right)^{\frac{1}{2}} = C_2 \left| (\xi_n) \right|_{\lambda^+}.
\end{align*}
It follows that $T$ defines a continuous linear operator from $\Lambda_1(\alpha)$ into $\mathcal{X}$. Moreover 
\[
    \|T((\xi_n))\|_k = \left( \sum_n |\xi_n|^2 e^{2r\alpha_n} \right)^{\frac{1}{2}} = |(\xi_n)|_r.
\]
Hence $T$ is a $(r,k)$-local imbedding.
\end{enumerate}
\end{proof}

\begin{xrem} \label{remark:1}
Contrary to the infinite case (\cite{AKT2}), the existence of an $(r,k)$-local imbedding from a finite type power series space into a graded Fréchet space need not imply the existence of an $(r^+, k)$-local imbedding for $r^+>r$. In fact it is not difficult to see that for a nuclear finite type power series space $\mathcal{X}=\Lambda_1(\alpha)$, there exists a $(r_1,r_2)$-local imbedding from $\Lambda_1(\alpha)$ into $\mathcal{X}$ if and only if $r_1 \leq r_2$. Indeed if there is a bounded set $B$ in $\mathcal{X}$ such that $d_n(B,U_{r_2}) \geq C e^{(r_1-1)\alpha_n}$ for some $C>0$, then for any $r_2<s<1$, there exists a $C_1>0$ such that $B \subseteq C_1 U_s$ which implies that 
\[
    d_n(B,U_r) \leq C_1 d_n(U_s,U_r) = C e^{(r_2-s)\alpha_n}
\]
for every $n$. Hence $r_1-1 \leq r_2-1$ or $r_1 \leq r_2$.

The condition in (iv) of the above Proposition can also be expressed by using the unit balls of the grading:
\end{xrem}

\begin{prop} \label{prop:1.3}
{\em 
There exists a $(r,p)$-local imbedding from $\Lambda_1(\alpha)$ into $\mathcal{X}$ if and only if
\[
    \sup_{k>p} \ \overline{\lim}_n \frac{\mathcal{E}_n(p,k)}{\alpha_n} \leq 1-r.
\]
}
\end{prop}

\begin{proof}
\ 
\begin{enumerate}
\item [($\Leftarrow$): ] Let $\rho$ and $r$ satisfy the given inequality. For a given $k>p$ in view of Proposition \ref{prop:1.2} (iii) we have;
\[
    \exists \ 0< \rho<1; \quad \overline{\lim}_n \frac{\mathcal{E}_n(p,k)}{\alpha_n} \leq \rho \  \overline{\lim}_n \frac{\mathcal{E}_n(p,k+1)}{\alpha_n} < 1-r.
\]
So for each $k>p$ we can find a $N(k) \in \mathbb{N}$, strictly increasing with respect to $k$, such that
\[
    -\ln d_n(U_k,U_p) < (1-r)\alpha_n.
\]
Let $\delta_1 = 1/2 e^{-(1-r)\alpha_{N(p+1)}}$. Since $U_{p+1}$ is precompact in $\mathcal{X}_p$ there exists a finite set $Z_1^1 \subseteq U_{p+1}$ such that;
\[
    d_n(U_{p+1}, U_{p}) \leq d_n(Z_1^1,U_{p})+\delta_1 \quad \forall \  n
\]
Hence
\[
    d_{N(p+1)}(Z_1^1,U_p) \geq e^{-(1-r)\alpha_{n_{N(p+1)}}} - \delta_1 = 1/2 e^{-(1-r)\alpha_{n_{N(p+1)}}}
\]
For each $n$, with $N(p+1) \leq n < N(p+2)$, we obtain, as above, a finite set $Z^1_n \subseteq U_{p+1}$ satisfying 
\[
    d_n(Z_n^1, U_p) \geq \frac{e^{-(1-r)\alpha_n}}{2},
\]
and set $Z^1 \doteq \bigcup_{N(p+1)\leq n < N(p+2)} Z_n^1$.

Continuing in this fashion we get finite sets $Z^s, \ s=1,2,\dots$ with $Z^s \subseteq U_{p+s}$ and $d_n(Z^s,U_p)\geq (e^{-(1-r)\alpha_n})/2$ for $N(p+s)\leq n<N(p+s+1)$. Note that $\bigcup_{s=1}^\infty Z^s$ is a bounded set since $\forall \ k=p+s, Z^{\bar{s}} \subseteq U_k, \bar{s} > s$. We can find a closed Hilbertian disc $B$ in $\mathcal{X}$, containing $\bigcup_{s=1}^\infty Z^s$ (see for example in Lemma 1.2 of \cite{V2}). For this disc and $n>N(p+1)$(say $N(p+s)\leq n<N(p+s+1)$) we have 
\[
    d_n(B,U_p)\geq d_n(Z^s,U_p)\geq \frac{e^{-(1-r)\alpha_n}}{2}
\]
So in view of Proposition \ref{prop:1.2} (iv) there is a (r,p)-local imbedding from $\Lambda_1(\alpha)$ into $\mathcal{X}$.

\item [($\Rightarrow$): ] This implication follows immediately from Proposition \ref{prop:1.2} (iv). Indeed in view of Proposition \ref{prop:1.2} (iv) there exists a bounded set $B$ and $C>0$ with $d_n(B,U_p) \geq C e^{-(1-r)\alpha_n},\ n=1,2,\dots$.
Hence $\forall\ k \ \exists \ C_i>0, \ i=1,2$;
\begin{align*}
    \mathcal{E}_n(p,k) & \leq -\ln d_n(B,U_p)+C_1 \\
                       & \leq (1-r)\alpha_n + C_2.
\end{align*}
So;
\[
    \overline{\lim}_n\  \frac{\mathcal{E}_n(p,k)}{\alpha_n} \leq 1-r \quad \forall \ k=p+1,\dots
\]
\end{enumerate}
\end{proof}

\begin{xrem} \label{remark:2}
Although the structural assumption $\underline{DN}$ on $\mathcal{X}$ is put into use in the Proposition above, the implication $(\Rightarrow)$ can be proved without any structural assumptions on $\mathcal{X}$. One can use the argument given in Proposition \ref{prop:1.1} of \cite{AKT1} to give a proof by just using the definition of local imbeddings.
\end{xrem}

\subsection{Approximate diametral dimension and local imbeddings}  \label{subs:1.3}

\ 

In this subsection we investigate the relationship between approximate diametral dimension of a nuclear Fréchet space $\mathcal{X}$ with properties $\underline{DN}$ and $\Omega$ and the existence of local imbeddings from $\Lambda_1(\mathcal{E}_n)$, the finite type power series space corresponding to the associated exponent sequence of $\mathcal{X}$. Since we wish to use the information obtained in the previous subsections, we are forced to put \emph{an additional assumption} on $\{\mathcal{E}_n\}_n$, namely we will assume that, $\lim_n \ln (n)/\mathcal{E}_n = 0$. This assumption guaranties the nuclearity of $\Lambda_1(\mathcal{E}_n)$ (\cite{D}), which will be needed in the sequel and also for the validity of the results of previous subsections.

\begin{defin} \label{def:3}
\cite{BPR}. \emph{The Approximate Diametral Dimension} of a Fréchet space $Y$, $\delta(Y)$, is defined as;
\begin{align*}
    \delta(Y) & \doteq \bigcup_{U \in \text{ zero neighborhoods of } Y} \  \bigcup_{B \in \text{ bounded subsets of } Y} \left\{ (t_n) : \lim_n \frac{t_n}{d_n(B,U)}=0 \right\} \\
    & = \{ (t_n)_n : \exists \text{ a neighborhood of zero } U \text{ and a bounded set } B \text{ of } Y \\
    & \text{ such that } \lim_n \frac{t_n}{d_n(B,U)}=0 \}.
\end{align*}
Approximate diametral dimension also admits a representation solely in terms of basis for the zero neighborhoods, $\{U_n\}_n$ of the Fréchet space $Y$ as; $\delta(Y)=\{(t_n):\  \exists \  s \ \forall \  k \geq s ; \  t_n / d_n(U_k, U_s) \rightarrow 0\}$ (See \cite{BPR}).

Now let us fix, for the rest of this subsection, a graded nuclear Fréchet space $\left\{ \mathcal{X}, \ ^{\|\ \|_k} \right\}$ with the properties $\underline{DN}$, $\Omega$, whose associated exponent sequence $(\mathcal{E}_n)$ satisfies $\left(\ln (n)/\mathcal{E}_n  \right)\rightarrow_n 0$, the grading $\{\|\ \|_k\}_k$ that is coming from the definition \ref{def:2}. Given $p<p+1<q$, in view of Proposition \ref{prop:1.1} (1), we have the inequality $d_n(U_q, U_{p+1})^{\frac{1}{d}} \leq C d_n(U_q, U_{p+1})$ for some $C>0$ and $0<\lambda<1$. So if a sequence $(t_n)_n$ satisfies 
\[
    \frac{|t_n|}{d_n(U_q,U_p)} \leq C \quad \forall n, \text{ for some } C 
\]
then
\[
    \frac{|t_n|}{d_n(U_q,U_{p+1})} \rightarrow 0
\]
Hence under our assumptions, we have 
\[
    \delta(\mathcal{X}) = \{ (t_n)_n: \quad \exists \ p; \forall \ q>p \ \sup_n |t_n|e^{\mathcal{E}_n(p,q)} < \infty\}.
\]
It is not difficult to show, by direct computation, that (\cite{BPR})
\[
     \delta(\Lambda_1(\alpha)) = \{(t_n)_n: \quad \overline{\lim}_n \frac{\ln |t_n|}{\alpha_n} < 0\}.
\]
\end{defin}

\begin{prop} \label{prop:1.4}
{\em 
    \[
    \delta(\mathcal{X}) \supseteq \delta (\Lambda_1(\alpha)) \Leftrightarrow \inf_p \sup_{q>p} \overline{\lim}_n \frac{\mathcal{E}_n(p,q)}{\alpha_n} = 0 
    \]
}
\end{prop}

\begin{proof}
\ 
\begin{enumerate}
\item [($\Rightarrow$): ] For a fixed pair $p<q$ let \[
    B(p,q) \doteq \{ (t_n) : \quad \sup_n |t_n| e^{\mathcal{E}_n(p,q)} < \infty \}.
\]
This linear space becomes a Banach space under $|(t_n)|_{pq} \doteq \sup_n |t_n| e^{\mathcal{E}_n(p,q)}$. Since $B(p,q^+)\subseteq B(p,q)$ for $p<q<q^+$, the space $B(p) \doteq \bigcap_{q>p} B(p,q)$ with the fundamental generating norms $\left\{|\cdot|_{pq}\right\}$ becomes a Fréchet space. Moreover by means of continuous inclusions $B(\bar{p}) \hookrightarrow B(p), \bar{p}<p$, we can put the inductive limit topology on $\delta(\mathcal{X}) = \bigcup_n B(p)$ and thus view $\delta(\mathcal{X})$ as a locally convex space which is an inductive limit of Fréchet spaces. Let us fix a $r, \  0<r<1$ and define
\[
    S_r \doteq (t_n): \quad \|(t)\| \doteq \sup_n \frac{|t_n|}{r^{\alpha_n}} < \infty.
\]
Clearly $S_r$ is a Banach space with respect to the norm $\|\cdot\|$, and $S_r \subset \delta(\Lambda_1(\alpha))$. By our assumption, $S_r \subset \delta(\mathcal{X})$ and since projections onto coordinates are continuous with respect to the inductive limit topology on $\delta(\mathcal{X})$, the inclusion $S_r \hookrightarrow \delta(\mathcal{X})$ is a sequentially closed linear operator. In view of Grothendieck Factorization theorem (\cite[p.~225]{Kothe1}), there is a $p(r)$ such that for every $q>p(r)$
\[
    \exists \ C>0 : \quad \sup_n |t_n| e^{\mathcal{E}_n(p(r),q)} \leq C \sup_n \frac{|t_n|}{r^{\alpha_n}}
\]
for any $(t_n) \in S_r$. In particular for every $q>p(r)$ there exists a $C>0: \  \mathcal{E}_n(p(r),q)\leq \ln C - \alpha_n \ln r$, which in turn implies that 
\[
    \inf_p \sup_{q>p} \overline{\lim}_n \frac{\mathcal{E}_n(p,q)}{\alpha_n} \leq -\ln r.
\]
Since this assertion holds for every $0<r<1$, we have 
\[
    \inf_p \sup_{q>p} \overline{\lim}_n \frac{\mathcal{E}_n(p,q)}{\alpha_n} = 0.
\]

\item [($\Leftarrow$): ] Fix an $r<1$ and choose a $(t_n)_n$ which satisfies
\[
    |t_n| \leq C r^{\alpha_n} \text{ for some } C>0 \text{ and } \forall \ n.
\]
In view of our assumption, we choose a $p$ so that $\sup_{q\geq p} \overline{\lim}_n \frac{\mathcal{E}_n(p,q)}{\alpha_n} < - \ln r$. It follows that there exists an $n_0$ such that for $n\geq n_0$;
\[
    e^{\mathcal{E}_n(p,q)} < \frac{1}{r^{\alpha_n}} \text{ for every } q\geq p.
\]
Hence, $\forall\  q>p \ \exists \ C_q>0$; $e^{\mathcal{E}_n(p,q)} \leq C_q/r^{\alpha_n}$; which in turn implies that $\forall \ q>p$; $\sup_n |t_n| e^{\mathcal{E}_n(p,q)} < +\infty$. It follows that:
\[
    \delta\left(\Lambda_1(\alpha)\right) \subseteq \bigcup_p \bigcap_{q>p} B(p,q) = \delta(\mathcal{X}).
\]
\end{enumerate}
\end{proof}

If we focus our attention to the associated exponent sequence $(\mathcal{E}_n)_n$ we have:

\begin{cor}
\[
    \delta(\mathcal{X}) = \delta(\Lambda_1(\mathcal{E})) \Leftrightarrow \inf_p \sup_{q>p} \overline{\lim}_n \frac{\mathcal{E}_n(p,q)}{\mathcal{E}_n} = 0.
\]
\end{cor}

\begin{proof}
Let us fix a $(t_n) \in \delta(\mathcal{X})$, and choose a $p$ such that \[
    \forall \ q>p; \quad \sup_n |t_n| e^{\mathcal{E}_n(p,q)} <+\infty.
\]
In view of Proposition \ref{prop:1.1} (4), there exists constants $C_1>0$, $C_2>0$ such that \[
    \mathcal{E}_n \leq C_1 \mathcal{E}_n(p,p+1)+C_2 \quad \forall \ n.
\]
It follows that;
\[
    \sup_n |t_n|e^{\frac{1}{C_1}\mathcal{E}_n} < \infty.
\]
Hence $\exists \ D>0$; 
\[
    \ln |t_n|+\frac{\mathcal{E}_n}{C_1} < D
\] 
or \[
    \overline{\lim} \frac{\ln|t_n|}{\mathcal{E}_n} \leq -\frac{1}{C_1} < 0.
\]
Hence we always have $\delta(\mathcal{X}) \subseteq \delta(\Lambda_1(\mathcal{E}))$. Now the corollary follows from Proposition \ref{prop:1.4}.
\end{proof}

We conclude this section with a summary theorem:
\begin{thm} \label{theorem:1}
Let $\mathcal{X}$ be a nuclear Fréchet space with the properties $\underline{DN}$ and $\Omega$. We also assume that the associated exponent sequence $(\mathcal{E}_n)_n$ of $\mathcal{X}$, is a nuclear exponent sequence of finite type (i.e. $\Lambda_1(\mathcal{E}_n)$ is nuclear). Then the following assertions are equivalent:
\begin{enumerate}
\item $\delta(\mathcal{X}) = \delta(\Lambda_1(\mathcal{E}))$
\item $\inf_{1\leq p < \infty} \sup_{q>p} \overline{\lim}_n \frac{\mathcal{E}_n(q,p)}{\mathcal{E}_n} = 0 $
\item There exists an $r$-local imbedding from $\Lambda_1(\mathcal{E})$ into $\mathcal{X}$ for every $0<r<1$.
\end{enumerate}
\end{thm}

\section{Tame Spaces, $\mathcal{X}$, with Properties $\underline{DN}$, $\Omega$ and $\delta(\mathcal{X})=\delta(\Lambda_1(\mathcal{E}))$}

To every continuous linear operator $T$ between two graded F-spaces $\{\mathcal{X}, \ ^{\|\ \|_k}\}_k$ and $\{\mathcal{Y}, \ ^{|\ |_k}\}_k$ one can associate a sequence of natural numbers, $\{\sigma_T(n)\}_n$, called the \emph{characteristic of continuity} of $T$, via
\[
    \sigma_T(n) \doteq \inf \left\{ s : \quad \exists \ C>0; \ |T(x)|_n \leq C \|x\|_s \forall \ x \in \mathcal{X} \right\}
\]

In general the characteristics of continuity of operators between graded Fréchet spaces could be very disorderly. However for certain pairs of Fréchet spaces control over the growth of characteristics of continuity can be obtained. For example, in the space of analytic functions on the unit disc, $\mathcal{O}(\Delta)$, with the grading $
\|f\|_k \doteq \sum_{n=0}^\infty \left|\frac{f^n(0)}{n!}\right|^2 e^{-n/k}, \quad k=1,2,\dots , \  f \in \mathcal{O}(\Delta)$, it is not difficult to see that;
\[
    \forall \ T \in \mathcal{L}(\mathcal{O}(\Delta), \mathcal{O}(\Delta)), \ \exists \ a \in \mathbb{N} \text{ such that } \sigma_T(n) \leq an, \quad \forall \ n=1,2,\dots
\]
Following (\cite{V2.1}, \cite{DF} c.f. \cite{KN}) we specify this property as:
\begin{defin} \label{def:4}
A pair of Fréchet spaces $\mathcal{X}$ and $\mathcal{Y}$ will be called a \emph{tame-pair} in case for a given pair of generating semi-norms $\{\|\ \|_k\}_k$ of $\mathcal{X}$ and $\{|\ |_k\}_k$ of $\mathcal{Y}$ there exists an increasing function $\psi: \mathbb{N} \rightarrow \mathbb{N}$ such that
\begin{align*}
\forall \ T \in \mathcal{L}(\mathcal{X},\mathcal{Y}), \quad \exists \ N \in \mathbb{N};\quad \sigma_T(n) \leq \psi(n),\quad \forall \ n \geq N.
\end{align*}
In case $\mathcal{X}=\mathcal{Y}$ we say that the Fréchet space $\mathcal{X}$ is \emph{tame}.
\end{defin}

We will use the symbol $(\mathcal{X},\mathcal{Y}) \in \mathcal{T}$ in case $(\mathcal{X},\mathcal{Y})$ form a tame-pair.

\begin{xrem} \label{remark:3}
\ 
\begin{enumerate}
\item The definition does not depend on the choice of semi-norms in the spaces $\mathcal{X}$ and $\mathcal{Y}$.
\item Plainly the definition is equivalent to the existence of a sequence, for a given pair of generating semi-norms $\{\|\ \|_k\}_k$ of $\mathcal{X}$ and $\{|\ |_k\}_k$ of $\mathcal{Y}$, $\{S_K\}_k$ of increasing functions $S_K: \ \mathbb{N} \rightarrow \mathbb{N}$, such that for every $T \in \mathcal{L}(\mathcal{X}, \mathcal{Y})$, there exists a $K$ with $\sigma_T(n) \leq S_K(n)$ \textit{for every} $n=1,2,\dots$.
\item The space $\mathcal{O}(\Delta)$ is tame. More generally every power series space of finite type is tame, as was observed by various authors (see 2.1 on \cite{DF} for details). A proof of this appears in 2.1 of \cite{KN}.
\item If $(\mathcal{X},\mathcal{Y})\in \mathcal{T}$, then the space of all continuous linear operators from $\{\mathcal{X}, \ ^{\|\ \|_k}\}$ into $\{\mathcal{Y}, \ ^{|\ |_k}\}$ admits a representation of the form 
\begin{align*}
    \mathcal{L}(\mathcal{X},\mathcal{Y})=\bigcup_{K=1}^\infty \bigcap_{n=1}^\infty \{ & T \in \mathcal{L}(\mathcal{X}, \mathcal{Y}): \\ 
   & \exists \ C>0 \ \|T(x)\|_n \leq C |x|_{S_K (n)}, \forall \ x \in \mathcal{X} \}
\end{align*}
Using this representation one can put a linear topology on $\mathcal{L}(\mathcal{X},\mathcal{Y})$ making it an LF-space (i.e. inductive limit of Fréchet spaces) by considering the semi-norms (2.1, \cite{DF}) $\|T\|_K^n \doteq \sup_{|x|_{S_K(n)}\leq1} \|T(x)\|_n$ on the space
\[
    \{ T  \in \mathcal{L}(\mathcal{X}, \mathcal{Y}):
    \exists C>0 \ \|T(x)\|_n \leq C |x|_{S_K (n)}, \forall x \in \mathcal{X} \}, K,n \in \mathbb{N}.
\]
therefore .....
\end{enumerate}
\end{xrem}

This structure allows one to use the results of well studied LF-spaces in the study of $\mathcal{L}(\mathcal{X}, \mathcal{Y})$ for $(\mathcal{X}, \mathcal{Y})\in \mathcal{T}$. These ideas are used in the study of nuclear Fréchet spaces $\mathcal{X}$ which form a tame pair with nuclear stable power series spaces of finite or infinite type in \cite{PISZCZEK} where a complete characterization of such spaces in terms of the linear topological invariants of Vogt are obtained. 

In this section we will, once again, consider nuclear Fréchet spaces, $\mathcal{X}$ with the properties $\underline{DN}$ and $\Omega$. We will assume that the associated exponent sequence of $\mathcal{X}$, $ \mathcal{E}=(\mathcal{E}_n)_n$ in addition to our usual assumption of being finitely nuclear is also \emph{stable}; i.e. $\sup_n (\mathcal{E}_{2n}/\mathcal{E}_n) <\infty$.

\begin{thm} \label{theorem:2}
Let $\mathcal{X}$ be a nuclear Fréchet space with stable finitely nuclear associated exponent sequence $\mathcal{E} =(\mathcal{E}_n)$ and which has properties $\underline{DN}$ and $\Omega$. Then $\mathcal{X}$ is isomorphic to a power series space of finite type if and only if $\mathcal{X}$ is tame and $\delta(\mathcal{X}) = \delta(\Lambda_1(\mathcal{E}))$.
\end{thm}

\begin{proof}
\ 
\begin{enumerate}
\item [($\Rightarrow$): ] If $\mathcal{X}$ is isomorphic to a power series space of finite type then $\mathcal{X}$ must be isomorphic to $\Lambda_1(\mathcal{E})$ by Proposition \ref{prop:1.1} of \cite{AKT2}. So $\delta(\mathcal{X}) = \delta(\Lambda_1(\mathcal{E}))$. Moreover $\mathcal{X}$ being a finite type power series space, is tame in view of Remark \ref{remark:3} (4) above.

\item[($\Leftarrow$): ] Let us fix a tame nuclear Fréchet space with stable finitely nuclear associated exponent sequence $\mathcal{E}=(\mathcal{E}_n)$, and which has properties $\underline{DN}$ and $\Omega$. In view of \cite{V2.1}, we can imbed $\mathcal{X}$ into $\Lambda_1(\mathcal{E})$ as a closed subspace. We will do so and consider the grading on $\mathcal{X}$ induced by $(\Lambda_1(\mathcal{E}), \ ^{|\ |_r})$. Throughout this proof we will, as usual, use the notation $\epsilon_n$ to denote the sequence $(0,\dots,0,1,0,\dots)$ where $1$ sits in the $n^{\text{th}}$ place, $n=1,2,\dots$. In view of Theorem \ref{theorem:1} above there exists an $r$-local imbedding from $\Lambda_1(\mathcal{E})$ into $\mathcal{X}$ for every $0\leq r < 1$. Let us fix an $r_0 <1$ and a corresponding $(r_0,r_{k_0})$, $r_{k_0} < 1$, local imbedding $T$, from $\Lambda_1(\mathcal{E})$ into $\mathcal{X}$. Say $|Tx|_{r_{k_0}} \geq C_0 |x|_{r_0}, \ \forall \ x \in \Lambda_1(\mathcal{E})$. Let $e_n \doteq (\epsilon_n/e^{r_0 \mathcal{E}_n}), \ n=1,2,\dots$ be the canonical  orthonormal basis of $\Lambda_{r_0}[\mathcal{E}_n]$, and set $g_n \doteq T(e_n), \ n=1,2,\dots$.

We note that $\{g_n\}_{n=1}^\infty$ is a finitely linearly independent sequence since $T$ is a local imbedding. We choose a sequence $\{f_n\}$ in $\mathcal{X}\subseteq\Lambda_1(\mathcal{E})$ satisfying;
\begin{enumerate}
\item [(1)] $f_n \in \text{span} \{g_1,\dots,g_{2n}\}, \ n=1,2,\dots$.
\item [(2)] $\langle f_n, f_s \rangle_{r_{k_0}} = 0, \ s=1,\dots, n-1; \ n=1,2,\dots$.
\item [(3)] $\langle f_n, \epsilon_k \rangle_{r_{k_0}} = 0, \ k=1,\dots, n; \ n=1,2,\dots$.
\item [(4)] $\langle f_n, f_n \rangle_{r_{k_0}} = 1, \ n=1,2,\dots$.
\end{enumerate}
where $\langle$, $\rangle_{r_{k_0}}$ is the inner product in $\Lambda_{r_{k_0}}[\epsilon]$.

Such a sequence exists and can be selected by induction since for each $n \in \mathbb{N}$, the space $\{g_1, \dots, g_{2n}\}$ is $2n$-dimensional and we impose $2n-1$ conditions for $f_n$.

Hence $f_n = \sum_{i=1}^{2n} c_i^n g_i$ for some scalars $\{c_i^n\}_i, \ n=1,2,\dots$ . We have;
\[
    1 = |f_n|_{r_{k_0}} \geq C_0 \left| \sum_{i=1}^{2n} c_i^n e_i \right|_{r_{0}} = C_0 \left( \sum_{i=1}^{2n} |c_i^n|^2 \right)^{\frac{1}{2}}, \   n=1,\dots
\]
So;
\begin{align}
    \sum_{i=1}^{2n} |c_i^n|^2 \leq \frac{1}{C_0}, \ \forall \ n \in \mathbb{N}.
\end{align}

Now fix an $s$, $0<s<1$, and estimate $|f_n|_s, \ n\in \mathbb{N}$ as:

\begin{align}
    |f_n|_s & = \left| \sum_{i=1}^{2n} c_i^n g_i \right|_s \leq \sum_{i=1}^{2n} |c_i^n| |g_i|_s \leq C_1 \sum_{i=1}^{2n} |c_i| |e_n|_{\sigma_T(s)} \nonumber \\
    & = C_1 \sum_{i=1}^{2n} |c_i^n| \left| \frac{\epsilon_i}{e^{r_0 \mathcal{E}_i}} \right|_{\sigma_T(s)} = C_1 \sum_{i=1}^{2n} |c_i^n| e^{(\sigma_T(s)-r_0)\mathcal{E}_i}
\end{align}
for some $C_1>0$, where $\sigma_T$ is the characteristic of continuity of $T$ with respect to the connonical gradings of $\Lambda_1(\mathcal{E}).$

Choose a $K(s)$, with $\max \{\sigma_T(s), r_0\} < K(s) < 1$, and proceed with the estimate (2.2),

\begin{align*}
    \sum_{i=1}^{2n} |c_i^n| \  & e^{(\sigma_T(s)-r_0)\mathcal{E}_i + K(s)\mathcal{E}_i - K(s)\mathcal{E}_i} \\
    \leq \  &  e^{(K(s)-r_0)\mathcal{E}_{2n}} \left( \sum_{i=1}^{2n} |c_i^n|^2 \right)^{\frac{1}{2}} \left( \sum_{i=1}^{2n} e^{2(\sigma_T(s)-K(s))\mathcal{E}_i} \right)^{\frac{1}{2}} \\
    \leq \  &  C e^{(K(s)-r_0)\mathcal{E}_{2n}},
\end{align*}
to obtain $\forall s<1$, $\exists C=C(S,T)$ and $K(s) < 1$ such that

\begin{align}
    |f_n|_s \leq C e^{(K(s)-r_0)\mathcal{E}_{2n}}, \quad n=1,2,\dots
\end{align}

We choose an increasing sequence $\{ K^+(s)\}_s$ with $K(s)<K^+(s)<1, \ \forall s<1$. 

On the other hand we also have, for each $n \in \mathbb{N}$, a representation of $f_n$ as;
\[
    f_n = \sum_{k=n+1}^\infty \beta_k^n \epsilon_k
\]
for some sequence $\{\beta_k^n\}_k$. For an $-\infty<r<r_{k_0}$ we estimate; 
\begin{align}
    |f_n|_r^2 & = \sum_{s=n+1}^\infty |\beta_s^n|^2 e^{2r\mathcal{E}_s} = \sum_{s=n+1}^\infty |\beta_s^n|^2 e^{2r_{k_0}\mathcal{E}_s} e^{2(r-r_{k_0})\mathcal{E}_s} \nonumber \\
    & \leq e^{2(r-r_{k_0})\mathcal{E}_{n+1}} |f_n|^2_{r_{k_0}} = e^{2(r-r_{k_0})\mathcal{E}_{n+1}}, \quad n=1,2,\dots.
\end{align}

In view of stability of $\{\mathcal{E}_n\}$ there exists a $C_0>0$ such that $C_0 \mathcal{E}_{2n} \leq \mathcal{E}_{n+1}, \ \forall n=1,2,\dots$. Hence proceeding with the estimate (2.4) we have;

\begin{align}
    |f_n|_r^2 \leq e^{2C_0(r-r_{k_0})\alpha_n} \quad \forall n \in \mathbb{N} \text{ and } -\infty<r<r_{k_0}.
\end{align}

We now choose and fix an $s_0$ with $-\infty < s_0 < -2/c_0$. In view of (2.5) we have;
\begin{align}
    |f_n|_{s_0} \leq e^{(C_0 s_0 - C_0 r_{k_0} + r_0)\mathcal{E}_{2n}} e^{-r_0\mathcal{E}_{2n}} \leq e^{-\mathcal{E}_{2n}} e^{-r_0\mathcal{E}_{2n}} \quad n=1,2,\dots.
\end{align}

We stress that $s_0$ depends only on the associated exponent sequence $\mathcal{E}$. Now for a sequence $\{\lambda_n\}_{n=1}^\infty$ and $s<1$, (2.3) above gives:
\[
    \sum_{i=1}^\infty |\lambda_i| |f_i|_s e^{r_0 \mathcal{E}_{2n}} \leq C \left( \sum |\lambda_i|^2 e^{2K^+(s)\mathcal{E}_{2n}} \right)^{\frac{1}{2}}
\]
for some $C=C(s)$.

It follows that the assignment that sends $\epsilon_n$ to $f_n e^{r_0 \mathcal{E}_{2n}}, \ n=1,2,\dots$, defines a continuous linear operator $\widehat{T}$ from $\Lambda_1(\mathcal{E}_{2n})$ into $\mathcal{X}$ that satisfies;

\begin{align*}
    |\widehat{T}(x)|_{r_{k_0}} = \left|\widehat{T}\left( \sum^\infty_{i=1} x_i \epsilon_i \right)\right|_{r_{k_0}} & = \left|\sum^\infty_{i=1} x_i f_i e^{r_0 \mathcal{E}_{2i}}\right|_{r_{k_0}} \\
    & = \left( \sum^\infty_{i=1} |x_i|^2 e^{2 r_0 \mathcal{E}_{2i}} \right)^\frac{1}{2} \\
    & = |x|_{r_0}, \quad \forall x \in \Lambda_1(\mathcal{E}_{2n}).
\end{align*}

Hence $\widehat{T}$ is an $(r_0, r_{k_0})$-local isomorphism from $\Lambda_1(\mathcal{E}_{2n})$ into $\mathcal{X}$.

Moreover $\widehat{T}$ has the additional property that it extends to be continuous from $\Lambda_0[(\mathcal{E}_{2n})]$ into $\Lambda_{s_0}[(\mathcal{E}_{n})]$. Indeed in view of (2.6), for $x \in \Lambda_{0}[(\mathcal{E}_{n})]$ we have;

\begin{align*}
    |\widehat{T}(x)|_{s_0} = \left|\widehat{T}\left( \sum_i x_i \epsilon_i \right)\right|_{s_0} & = \left| \sum^\infty_{i=1} x_i f_i e^{r_0 \mathcal{E}_{2i}} \right|_{s_0} \\
    & \leq \sum^\infty_{i=1} |x_i| |f_i|_{s_0} e^{r_0 \mathcal{E}_{2i}} \leq \sum^\infty_{i=1} |x_i| e^{-\mathcal{E}_{2i}} \\
    & \leq \left( \sum^\infty_{i=1} e^{-2 \mathcal{E}_{2i}} \right)^\frac{1}{2} |x|_0.
\end{align*}

We now vary $r<1$, and obtain a family $\{\widehat{T}_r\}$ of $(r,r_{k_r})$-local imbeddings from $\Lambda_1\big((\mathcal{E}_{2n})\big)$ into $\mathcal{X}$ with the additional property mentioned above.

Using the elementary inequality 
\[
    |\ |_{t_2} \leq |\ |_{t_1}^{\frac{t_3-t_2}{t_3-t_1}} |\ |_{t_3}^{\frac{t_2-t_1}{t_3-t_1}} ,
\]
for a given $t_1<t_2<t_3$, which is valid in any power series space, for any $s_0<s<1$ we choose a $\rho(s)<1$ such that
\begin{align}
    \forall \ 0\leq r < 1 \ \exists C > 0: \quad |\widehat{T}_r(x)|_s \leq C |x|_{\rho(s)}
\end{align}
After these preparations we now proceed to show that $\Big(\mathcal{X}, \Lambda_1\big((\mathcal{E}_{2n})\big)\Big) \in \mathcal{T}$.

According to our assumption, there exists a sequence $\{S_\alpha\}_{\alpha=1}^\infty$ of increasing functions from $\mathbb{N}$ into $\mathbb{N}$, such that $\forall T \in \mathcal{L}(\mathcal{X}, \mathcal{X})$ and $\exists \alpha \in \mathbb{N}$ such that $\sigma(T)\leq S_\alpha$.

Let now $S$ be a given continuous linear operator from $\mathcal{X}$ into $\Lambda_1\big(\mathcal{E}_{2n})\big)$. In view of (2.7) above, the family, $\{\widehat{T}_r \circ S\}_{r<1}$ of continuous linear operators from $\mathcal{X}$ into $\mathcal{X}$ satisfy;
\[
    \sigma_{\widehat{T}_r \circ S} \leq \sigma_S \circ \rho. 
\]
Hence this family is in $F \doteq \{U \in \mathcal{L}(\mathcal{X}, \mathcal{X}): \quad \sigma_U \leq \sigma_{S} \circ \rho \}$. On $F$ we consider the topology coming from the semi-norms $\{\|\cdot\|_s^{\sigma_{S} \circ \rho(s)}\}_{s=1}^\infty$, and on $\mathcal{L}(\mathcal{X}, \mathcal{X}) = \bigcup_{\alpha=1}^\infty \bigcap_{s=1}^\infty \{U \in \mathcal{L}(\mathcal{X}, \mathcal{X}): \ \|U\|_s^{S_\alpha(s)} < \infty \}$, we consider the LF-space structure as explained in Remark \ref{remark:3} (4) above. Since evaluation at points of $\mathcal{X}$ are continuous in both $F$ and $\mathcal{L}(\mathcal{X}, \mathcal{X})$, the inclusion $F \subseteq \mathcal{L}(\mathcal{X}, \mathcal{X})$ has a sequentially closed graph. It follows that there exists $\alpha$ such that $F \subseteq \bigcap_{s=1}^\infty \{U \in \mathcal{L}(\mathcal{X}, \mathcal{X}): \ \|U\|_s^{S_{\alpha(s)}} < \infty\}$ in view of Grothendieck factorization theorem \cite[p.~68]{Kothe2}. It follows that there exists an $\alpha \in \mathbb{N}$ such that;
\[
    \sigma_{\widehat{T}_r \circ S} \leq S_\alpha \quad \forall r<1.
\]
In particular for each $r<1$ there exists a $\widehat{C}>0$ such that;
\begin{align}
    |S x|_r \leq C \left|\widehat{T}_r(Sx)\right|_{r_{k_r}} \leq \widehat{C} |x|_{S_\alpha (r_{k_r})}.
\end{align}
Now if we set $\widehat{S}_\alpha (r) \doteq S_\alpha(r_{k_r})$, $r<1$, $\alpha \in \mathbb{N}$, the analysis above shows that $\forall S \in \mathcal{L}\Big(\mathcal{X}, \Lambda_1\big((\mathcal{E}_{2n})\big)\Big) \  \exists \ n \in \mathbb{N}$ such that $\sigma_S \leq \widehat{S}_\alpha$. Hence it follows that $\Big(\mathcal{X}, \Lambda_1\big((\mathcal{E}_{2n})\big)\Big) \in \mathcal{T}$.

Now Theorem 11 of \cite{PISZCZEK} implies that $\mathcal{X}$ satisfies the strong $\Omega$ condition, $\bar{\Omega}$, of Vogt. This together with our assumption that, $\mathcal{X}$ has the property $\underline{DN}$ allows us to conclude that $\mathcal{X}$ is isomorphic to a finite type power series space (Proposition 2.9.18 of \cite{M-V}). Hence $\mathcal{X} \cong \Lambda_1 (\mathcal{E})$ (Proposition 1.1 of \cite{AKT2}).
\end{enumerate}
\end{proof}

\vspace{0.5cm}

\section{Spaces of Analytic Functions}

In this section we will focus our attention to a particular class of nuclear Fréchet spaces with the properties $\underline{DN}$ and $\Omega$, namely the spaces of analytic functions on Stein manifolds. Stein manifolds, being closed connected submanifolds of complex euclidean spaces $\mathbb{C}^N$, possess a rich supply of analytic functions. These spaces, with the usual topology of uniform convergence on compact subsets, form an important subclass of Fréchet spaces with the properties $\underline{DN}$ and $\Omega$. The linear topological properties of $\mathcal{O}(M)$, the Fréchet space of analytical functions on a Stein manifold $M$, and the complex analytic properties of $M$ that are reflected on $M$ by the type of $\mathcal{O}(M)$, have been studied by several authors (see \cite{AYI} \cite{Zah1} and the references therein). In this context, we show, in subsection \ref{subs:3.1} that local imbedding of finite type power series spaces into $\mathcal{O}(M)$ can be used to construct Green's functions in $M$. In subsection \ref{subs:3.2} we classify Stein manifolds $M$, for which $\mathcal{O}(M)$ is tame. Some results and concepts from Pluripotential Theory will be used in the sequel. For undefined terminology and background in these matters we refer the reader to \cite{KLIMEK}.

\subsection{Local imbeddings of finite type power series spaces into $\mathcal{O}(M)$ and Green's functions} \label{subs:3.1}

\ 

Let $M$ be a complex manifold and fix a point $z_0 \in M$. We will use the symbol $PSH(M)$ to denote the set of all plurisubharmonic functions on $M$. Employing norms in a local chart centered at $z_0$ we consider
\[
    \mathcal{L}_{z_0} \doteq \{u \in PSH(M), \ u \leq 0 \text{ and } u(z)-\ln\|z-z_0\| \text{ is bounded near } z_0\}
\]
and set 
\[
    g(\xi, z_0) \doteq \sup\{u(\xi): \ u \in \mathcal{L}_{z_0}\} .
\]
This assignment, if not $\equiv - \infty$, defines a plurisubharmonic function on $M$. We will call $g(\cdot z_0)$ the \emph{Green's function} of $M$ with pole at $z_0$, and say that Green's function with pole at $z_0$ exists in case $g(\cdot, z_0)$ is not identically equal to $-\infty$. If $M$ is parabolic i.e. if it does not possess a nonconstant bounded plurisubharmonic function then of course, no Green's function exist. In one variable nonparabolicity characterizes existence of Green's functions, \cite{S-N}, however in several complex variables there is no such general result. The difficulty seems in constructing a negative plurisubharmonic function with a pole at the given point from the class of negative plurisubharmonic functions.

\begin{prop} \label{prop:3.1}
{\em
Let $M$ be a Stein manifold of dimension $d$ and let $z_0$ be a point in $M$. If there exists a local imbedding from $\Lambda_1(\alpha)$ into $\mathcal{O}(M)$, for some finitely nuclear exponent sequence $\{\alpha_n\}_n$, with $\underline{\lim} \  n^{1/d}/\alpha_n > 0$, then Green's function with a pole at $z_0$ exists. 
}
\end{prop}

\begin{proof}
Let us fix $M$ and the point $z_0$. In view of \cite{S-F} there exists a local biomorphism $\Phi: \ \Delta^d_e \twoheadrightarrow M$, from the polydisc with centre $0$ and radius $e$ of $\mathbb{C}^d$ onto $M$ such that $\Phi(0)=z_0$. We will use $\Phi$ to imbed $\mathcal{O}(M)$ into $\mathcal{O}(\Delta^d_e)$ via the composition operator $f \mapsto f\circ \Phi$, $f \in \mathcal{O}(M)$ (\cite{AT1}). We will choose a bijection $\rho: \mathbb{N} \rightarrow \mathbb{N}^d$ such that $|\rho(n)|$ is strictly increasing with $n$ and is ordered lexicographic on the level sets of $|\rho(\cdot)|$, and use it to define an isomorphism between $\Lambda_1(n^{1/d})$ and $\mathcal{O}(\Delta^d_e)$ by the correspondence $\epsilon_i \rightarrow z^{p(i)} = z_1^{p_1(i)} \dots z_d^{p_d(i)}, \ i=1,2,\dots$. We note that there exists constants $\beta_1 > 0$ and $\beta_2>0$ such that $\beta_1(\rho(n)) \leq n^{1/d} \leq \beta_2(\rho(n)), \ n=1,2,\dots$ (see e.g. \cite{R}).

Let us fix an $(r_0, k_0)$-local imbedding, $T$, from $\Lambda_1(\alpha_n)$ into $\Lambda_1 (n^{1/d})$, $T(\Lambda_1(\alpha_n)) \subseteq \mathcal{O}(M) \subseteq \Lambda_1(n^{1/d})$, which exists in view of our assumptions. Let
\[
f_n \doteq T \left(\frac{\epsilon_n}{e^{r_0\alpha_n}}\right), \ n=1,2,\dots. 
\]
The sequence $\{f_n\}_n$ is finitely linearly independent, so we can, choose a sequence $\{g_n\}^\infty_{n=1}$ of elements of $\mathcal{O}(M)$ with
\begin{enumerate}
\item $g_n \in \text{span} \{f_1,\dots,f_n\}, \ n=1,2,\dots$
\item $\langle g_n, \epsilon_i \rangle_{k_0}$, $i=1,\dots,n-1$ where $\langle$, $\rangle_{k_0}$ is the inner product in $\Lambda_{k_0}[n^{1/d}] \  \forall \ n=1,2,\dots$ 
\item $|g_n|_{k_0} = 1$, $n=1,2,\dots$
\end{enumerate}
Note that if $g_n = \sum^n_{k=1} c_k^n f_k$ then we estimate, as in the above arguments;
\begin{align*}
    \sum^n_{k=1} |c_k^n|^2 = \sum^n_{k=1} \left| \frac{c_k^n}{e^{2r_0\alpha_k}} \right|^2 e^{2r_0\alpha_k} = \left| \sum^n_{k=1} c_k^n \frac{\epsilon_k}{e^{2r_0\alpha_k}}\right|^2_{r_0} & \leq C \left| T \left( \sum^n_{k=1} c_k^n \frac{\epsilon_k}{e^{2r_0\alpha_k}}\right) \right|^2_{k_0} \\
    & = C |g_n|^2_{k_0} \leq C,
\end{align*}
for some $C>0$ and for all $n=1,2,\dots$.

Hence, for a given $r<1$;
\begin{align} \label{eqn:3.4}
|g_n|_r & = \left| \sum_{i=1}^n c_i^n f_i \right|_r \leq C_1 \sum_{i=1}^n |c_i^n| \left| \frac{\epsilon_i}{e^{r_0\alpha_n}} \right|_{\sigma_T(r)} \nonumber \\
& = \sum_{i=1}^n |c_i^n| e^{(\sigma_T(r)-r_0)2n} \leq C_2 e^{\rho(r)\alpha_n}, \quad n=1,2,\dots
\end{align}
and for some constants $C_1>0$, $C_2>0$ where $\sigma_T(r)-r_0 \leq \rho(r) < 1$ is a chosen number that depends on $r$ (and $T$).

On the other hand in view of equation \eqref{eqn:3.2} above, each $g_n,\ n=1,2,\dots$ has an expansion as \[
    g_n = \sum_{i \geq n} d_i^n \epsilon_i \leftrightarrow \sum_{i \geq n} d_i^n z^{\rho(i)}
\]
in $\Lambda_1(n^{1/d})$ and in $\mathcal{O}(\Delta^d_e)$, respectively. By abuse of notation we will think of $g_n \in \mathcal{O}(M) \subseteq \mathcal{O}(\Delta_e^d), \ n=1,2,\dots$ as an analytic function on $\Delta_e^d$ with Taylor series;
\begin{align} \label{eqn:3.5}
    g_n(z) = \sum_{|t|\geq \frac{1}{\beta_2} n^{1/d}} \alpha_t^n z_1^{t_1} \dots z_d^{t_d}, \quad n=1,2,\dots.
\end{align}

Choose $h_n \in \mathcal{O}(M)$ such that $g_n = h_n \circ \Phi, \ n=1,2,\dots$. For a given compact set $K \subseteq M$ choose $r<1$ so that $\Phi(\Delta^d_{e^r}) \supseteq K$. So in view of equation \eqref{eqn:3.4};
\[
    \sup_{z \in K} |h_n(z)| \leq \sup_{\xi \in \Delta_{e^r}^d} |g_n(\xi)| \leq C |g|_{r^+} \leq C C_2 e^{\rho(r^+)\alpha_n} 
\]
for some constant $C>0$ and a choice of $r<r^+<1$.

Hence the function
\[
    u(z) \doteq \overline{\lim_{\xi \rightarrow z}}\ \overline{\lim}_n \frac{\ln |h_n(\xi)|}{\alpha_n},\quad z \in M
\]
defines a plurisubharmonic function on $M$ that is bounded by $1$. This plurisubharmonic function is not $\equiv -\infty$. To see this let us assume, with the anticipation of getting a contradiction that it is identically equal to $-\infty$. Fix large $k_0<r<1$ such that $|x|_{k_0} \leq C \sup_{z\in \overline{\Delta}_{e^r}^d} |x(z)|, \ \forall x \in \mathcal{O}(\Delta^d_e)$ and choose a compact set $K \subset M$ that contains $\Phi(\overline{\Delta}_{e^r}^d)$.

Our assumption and Hartogs theorem (\cite[p.~70]{KLIMEK}) gives for each $N\in \mathbb{N}$, an $n_0 \in \mathbb{N}$ such that:
\[
    \sup_{z \in K} |h_n(z)| \leq e^{-N\alpha_n}, \quad n\geq n_0.
\]

This in turn gives that for some $C>0$
\[
    1 = |g_n|_{r_0} \leq C e^{-N\alpha_n}, \quad \text{ for } n \geq n_0.
\]
So $u$ is not identically $-\infty$.

Now consider a $z \in M$ near $z_0$, say with $\|z-z_0\| = e^r$ for some very large and negative $r$. Since $\Phi$ is a local biholomorphism, there exists a $C_0>0$, independent of $r$, and $\xi \in \Delta_e^d$ with $\|\xi\| < C_0 e^r$ such that $\phi(\xi) = Z$. Using \eqref{eqn:3.5} we estimate, with, $k_0^- < k_0$ and $C_1>0$,
\begin{align*}
    |g_n(\xi)| & \leq \sum_{|t| \geq C n^{1/d}} |\alpha_t^n| C_0^{|t|} e^{r|t|} e^{-k_0^-} |t| e^{k_0^- |t|} \\
    & \leq C_0 e^{C(r-k_0^-)n^{1/d}} \left(\sum_{|t|} |\alpha_t^n|^2 e^{2k_0 |t|}\right)^{1/2} \leq C_1 e^{C(r-k_0^-)n^{1/d}}, \quad n=1,2,\dots
\end{align*}

Hence our assumption on $(\alpha_n)$ implies $\exists \  C_2>0$, $C_3>0$:
\[
    \overline{\lim}_n \frac{\ln |h_n(z)|}{\alpha_n} \leq \overline{\lim}_n C(r-k_0^-) \frac{n^{1/d}}{\alpha_n} \leq C_2 \ln \|z-z_0\|+C_3.
\]
So $\frac{u}{C_2}$ has a logarithmic singularity at $z_0$, and is a bounded plurisubharmonic function. It follows that $g_m(\cdot,z_0)$ is not identically $-\infty$. 
\end{proof}

\subsection{Diametral dimensions of analytic function spaces} \label{subs:3.2}

\ 

In this subsection we will investigate the diametral dimension of the spaces of analytic functions on Stein manifolds. These invariants for Fréchet spaces are in a sense dual to approximate diametral dimensions, however they are more extensively studied.

For a nuclear Fréchet space $\mathcal{X}$ with a neighborhood basis of $0$, $\{U_p\}_{p=1}^\infty$, consisting of discs, $\Delta(\mathcal{X})$, is defined by (in the notation of Section 1):
\begin{align*}
    \Delta(\mathcal{X}) & = \{(t_n);\ \forall p \ \exists q:\quad t_n d_n(U_q,U_p)\rightarrow 0\}. \\
                        & = \{(t_n);\ \forall p \ \exists q:\quad t_n e^{-\mathcal{E}_n(p,q)}\rightarrow 0\}.
\end{align*}
As the notation suggests, it is easy to see that, this sequence space does not depend upon the neighborhood basis chosen and is an invariant of the Fréchet space $\mathcal{X}$.

In this subsection we will use a spacial generating norm-system for the analytic function spaces, unless stated otherwise. To describe these norms, suppose a Stein manifold $M$ of dimension $d$ is given. We choose a strictly plurisubharmonic $C^\infty$ exhaustion function, $\rho$, of $M$ and consider a sequence $r_p \uparrow \infty$ such that the sub-level sets $D_p \doteq \{z \in M; \ \rho(z)<r_p\},\ p=1,2,\dots$, are strictly pseudoconvex. We set $K_p\doteq \bar{D}_p, \ p=1,2,\dots$. Following (\cite{Zah2}, c.f. \cite{Zah1}) we choose for each $p$, a Hilbert space, $H_p$, that satisfies;
\[
    A(K_p) \hookrightarrow H_p \hookrightarrow AC(K_p) \hookrightarrow \mathcal{O}(D_p)
\]
where $A(K_p)$ denotes the germs of analytic functions on the compact set $K_p$ with the inductive limit topology, $AC(K_p)$ denotes the Banach space that is the closure of $A(K_p)$ in $C(K_p)$, the Banach space of continuous functions on $K_p$ with the sup-norm, and the $\hookrightarrow$'s are imbeddings with dense range, $p=1,2,\dots$. The norms that we will use to generate the topology $\mathcal{O}(M)$ will be the hilbertian norms of $H_p$'s, $\{\|\ \|_p\}_p$. Their corresponding unit balls in $\mathcal{O}(M)$ will be denoted by $U_p$, $p=1,2,\dots$. The primary reason of our usage of these semi-norms is because of the beautiful formula of Nivoche, Poletsky, Zaharyuta:

\begin{align} \label{eqn:3.1}
\forall p<q \quad \lim_n \frac{\mathcal{E}_n(p,q)}{\left(\frac{2\pi d! n}{\tilde{C}(K_p,D_q)}\right)^\frac{1}{d}} = 1
\end{align}
where $\tilde{C}(K_p,D_q)=\sup\{\int_{K_p} (d d^c u)^n:\ u \in PSH(D_q)-1\leq u \leq 0\}$ (Proposition 4.6, 4.8, 4.12 on \cite{Zah2}). We refer the reader to \cite{Zah2} and \cite{Zah1} for a proof and a discussion about the history of this formula.

An immediate consequence of equation \eqref{eqn:3.1} is;

\begin{prop} \label{prop:3.2}
Let $M$ be a Stein manifold of dimension $d$. Then 
\[
    \Delta(\mathcal{O}(M)) \equiv \{(t_n):\ \forall p \ \exists q \text{ such that } \sup_n |t_n| e^{-\alpha_n/c(p,q)} < \infty\}
\]
where for $p<q$, $c(p,q) \doteq \tilde{C}(K_p,D_q)^\frac{1}{d}$, and $\alpha_n \doteq (2\pi d!)^\frac{1}{d} n^\frac{1}{d}, \ n=1,2,\dots$.
\end{prop}

\begin{proof}
\ 
\begin{enumerate}
\item [($\subseteq$): ] Choose a $(t_n)_n \in \Delta(\mathcal{O}(M))$. Fix $p$ and choose $q$ such that 
\[
    \sup_n |t_n| e^{-\mathcal{E}_n(p,q)} < +\infty.
\]
Now in view of the (2') of Proposition \ref{prop:1.1} there exists a $q^+$, $C>0$ and $0<\lambda<1$ such that;
\[
    \mathcal{E}_n(p,q) \leq (1-\lambda) \mathcal{E}_n (p,q^+)+C.
\]

For $\varepsilon>0$ such that $(1+\varepsilon) < \frac{1}{1-\lambda}$ in view of \eqref{eqn:3.1} there exists an $N$ such that
\[
    (1-\varepsilon)\frac{\alpha_n}{C(p,q^+)} \leq \mathcal{E}_n(p,q^+) \leq (1+\varepsilon)\frac{\alpha_n}{C(p,q^+)}, \quad n \geq N.
\]
Hence we have;
\begin{align*}
    |t_n| \leq C_1 e^{\mathcal{E}_n(p,q)} & \leq C_2 e^{(1-\lambda)\mathcal{E}_n(p,q^+)} \\ & \leq e^{(1-\lambda)(1+\varepsilon)\frac{\alpha_n}{C(p,q^+)}} \leq e^\frac{\alpha_n}{C(p,p^+)}, \quad n \geq N.
\end{align*}

\item [($\supseteq$): ] Choose a sequence $(t_n)$ from the right hand side. For a fixed $p$, choose a $q$ such that $\sup_n |t_n| e^{-\alpha_n/c(p,q)} < \infty$. By (2') of Proposition \ref{prop:1.1} choose $k$, $0<\rho<1$ and $C>0$ so that $\mathcal{E}_n(p,q) \leq (1-\rho)\mathcal{E}_n (p,k)+C$. Let $\varepsilon>0$ be so that $1-\rho < 1-\varepsilon$ and choose an $N$ (from \eqref{eqn:3.1})
\[
    \frac{-(1+\varepsilon)\alpha_n}{C(p,q)} \leq -\mathcal{E}_n(p,q) \leq -\frac{(1-\varepsilon)\alpha_n}{C(p,q)}, \quad n \geq N
\]
Hence 
\[
    |t_n|^{(1-\varepsilon)} e^{-\mathcal{E}_n(p,q)} \leq \left( |t_n| e^{\frac{-\alpha_n}{C(p,q)}} \right)^{1-\varepsilon} \leq Q < \infty.
\]
for some $Q>0$. Moreover since $-\mathcal{E}_n(p,q) \geq -(1-\varepsilon)\mathcal{E}_n(p,k)-C$ we have;
\[
    \sup_n\left( |t_n| e^{-\mathcal{E}_n(p,k)} \right)^{1-\varepsilon} \leq Q e^C \Rightarrow \sup_n |t_n| d_n(U_k,U_p) < \infty.
\]
It follows that $(t_n) \in \Delta(\mathcal{O}(M))$.
\end{enumerate}
\end{proof}

\begin{cor} \label{corollary:3.2}
Let $M$ be a Stein manifold of dimension $d$. Then \[
    \Delta(\mathcal{O}(M)) = \Delta(\Lambda_1(n^\frac{1}{d})) \Leftrightarrow \inf_p \sup_{q\geq p} \frac{1}{C(p,q)} = 0 \Leftrightarrow \delta(\mathcal{O}(M)) = \delta(\Lambda_1(n^\frac{1}{d})).
\]
\end{cor}

\begin{proof}
Suppose that $\Delta(\mathcal{O}(M)) = \Delta(\Lambda_1(n^\frac{1}{d}))$. Suppose that there is an $\varepsilon>0$ such that $\inf_p \sup_{q\geq p} \frac{1}{c(p,q)} > \varepsilon > 0$. In view of equation \eqref{eqn:3.1} this implies,
\[
    \inf_p \sup_{q\geq p} \lim_n \frac{\mathcal{E}_n(p,q)}{n^\frac{1}{d}} > \varepsilon(d!2\pi)^\frac{1}{d}
\]
Setting $\kappa \doteq \varepsilon(d!2\pi)^\frac{1}{d}$, we have;
\[
    \forall p \ \exists q \text{ and } N \text{ such that } \frac{1}{d_n(U_q,U_p)} \geq e^{\kappa n^\frac{1}{d}}, \quad n \geq N.
\]
Hence, 
\[
    e^{\kappa n^\frac{1}{d}} \in \Delta(\mathcal{O}(M)) = \Delta(\Lambda_1(n^\frac{1}{d})) = \{(\xi_n):\ \forall\ r<1, \quad \lim_n |\xi_n| r^{n^{1/d}} = 0\}.
\]
This contradiction shows that $\inf_p \sup_{q\geq p} 1/c(p,q) = 0$.

On the other hand if $\inf_p \sup_{q\geq p} 1/c(p,q) = 0$, and $r<1$ is given, choose a $p$ such that, in the notation of Proposition \ref{prop:3.1},
\[
    \sup_{q\geq p} \frac{\alpha_n}{c(p,q)n^\frac{1}{d}} \leq -\ln r \quad \forall n.
\]
In view of Proposition \ref{prop:3.1}, for $p$ there is a $q>>p$ such that $\sup_n |t_n| e^{\frac{-\alpha_n}{c(p,q)}} < \infty$. Hence $\sup_n |t_n| r^{n^\frac{1}{d}} < \infty$. This shows that 
\[
(t_n) \in \{(\xi_n):\ \forall r<1, \ \lim_n|\xi_n|r^{n^\frac{1}{d}} = 0\} = \Delta(\Lambda_1(n^\frac{1}{d})).
\]
So $\Delta(\mathcal{O}(M)) \subseteq \Delta(\Lambda_1(n^\frac{1}{d}))$.

On the other hand, $(n^\frac{1}{d})_n$, being the associated exponent sequence of $\mathcal{O}(M)$, always satisfies $\Delta(\Lambda_1(n^\frac{1}{d})) \subseteq \Delta(\mathcal{O}(M))$ (Proposition 1.1, \cite{AKT2}). So we have $\Delta(\Lambda_1(n^\frac{1}{d})) =\Delta(\mathcal{O}(M))$. Since $\mathcal{O}(M)$ is isomorphic to a closed subspace of $\Lambda_1(n^{1/d})$, the other equivalence directly follows from Proposition \ref{prop:1.4} and the fact that approximate diametral dimension of a nuclear Fréchet space is greater than the approximate diametral dimension of its subspaces (\cite{BPR}).
\end{proof} %of corollary

\begin{thm} \label{theorem:3}
Let $M$ be a Stein manifold of dimension $d$. Then either $\Delta(\mathcal{O}(M)) = \Delta(\mathcal{O}(\Delta^d))$ or $\Delta(\mathcal{O}(M)) = \Delta(\mathcal{O}(\mathbb{C}^d))$. 
\end{thm}

\begin{proof}
    Let us fix a Stein manifold $M$ with dimension $d$. We will use, in what follows, the grading on $\mathcal{O}(M)$ described in the beginning of the section. Using the notation above, for a given $p\geq 1$, we set, $c(p) \doteq \lim_{q>p} c(p,q) = \inf_{q>p} c(p,q)$. Since $c(\cdot,\cdot)$ is increasing in the first variable, the sequence $\{c(p)\}$ is an increasing sequence. We have two cases:
\begin{enumerate}
\item [1.] $c(p)$ is zero for all $p\geq 1$.
    Suppose this is the case. Let us choose a $(t_n)_n \in \Delta(\mathcal{O}(\mathbb{C}^d))=\Delta(\Lambda_\infty(n^\frac{1}{d}))$, and choose an $R>1$ and $C>0$, so that $|t_n| \leq C R^{n^\frac{1}{d}}, \ \forall n$. For a given $p$ let us choose a $q$ such that $c(p,q) \leq \frac{1}{\ln (\tilde{R})}$ for some $\tilde{R}$ with $R^{1/(2\pi d!)^\frac{1}{d}} << \tilde{R}$. Then 
\[
    |t_n| e^{-\frac{\alpha_n}{c(p,q)}} \leq |t_n| e^{-\alpha_n \ln \tilde{R}+ n^{1/d}\ln R -n^{1/d}\ln R} \leq C.
\]
So $(t_n)_n \in \Delta(\mathcal{O}(M))$ in view of Proposition \ref{prop:3.1}. Hence 
\[
\Delta(\Lambda_\infty(n^{1/d})) \subseteq \Delta(\mathcal{O}(M)).
\]
However, $\{n^\frac{1}{d}\}_n$, being the associated exponent sequence of $\mathcal{O}(M)$, the inclusion $\Delta(\mathcal{O}(M)) \subseteq \Delta(\Lambda_\infty(n^{1/d}))$ is always true. It follows that in this case $\Delta(\mathcal{O}(M)) = \Delta(\mathcal{O}(\mathbb{C}^d))$.

\item [2.] The sequence $\{c(p)\}$ increases to a non-zero number $c$.
Suppose this is the case. We have two possibilities either $c\in \mathbb{R}$ or $c=\infty$.

Let us suppose that $c>0$ is a real number. Fix natural numbers $\alpha<\gamma<\beta$. For given natural numbers $n$ and $m$, plainly $d_{n+m}(U_\beta,U_\alpha) \leq d_n(U_\beta,U_\gamma)d_m(U_\gamma,U_\alpha)$, where $U_t$'s, $1\leq t < \infty$, are the unit balls corresponding to our grading. In particular we have 
\begin{align} \label{eqn:3.2}
\mathcal{E}_{n+m}(\alpha,\beta) \geq \mathcal{E}_n(\gamma,\beta) + \mathcal{E}_m(\alpha,\gamma); \quad n,m\geq 1
\end{align}

%start of 0965_001
Fix an $\varepsilon >0$ and using \eqref{eqn:3.1}, choose $N=N(\alpha,\gamma,\beta)$ such that
\begin{align*}
    (1-\varepsilon) \frac{\alpha_{n+m}}{c(\alpha,\beta)} & \leq \mathcal{E}_{n+m}(\alpha,\beta) \leq (1+\varepsilon) \frac{\alpha_{n+m}}{c(\alpha,\beta)} \\
    (1-\varepsilon) \frac{\alpha_{n}}{c(\gamma,\beta)} & \leq \mathcal{E}_{n}(\gamma,\beta) \leq (1+\varepsilon) \frac{\alpha_{n}}{c(\gamma,\beta)} \\
    (1-\varepsilon) \frac{\alpha_{m}}{c(\alpha,\gamma)} & \leq \mathcal{E}_{m}(\alpha,\gamma) \leq (1+\varepsilon) \frac{\alpha_{m}}{c(\alpha,\gamma)}, \text{ for } n,m \geq N 
\end{align*}
Hence we have
\[
    (1-\varepsilon) \frac{m^\frac{1}{d}}{c(\alpha,\gamma)} + (1-\varepsilon) \frac{n^\frac{1}{d}}{c(\gamma,\beta)} \leq (1+\varepsilon) \frac{(n+m)^\frac{1}{d}}{c(\alpha,\beta)}
\]
for $n,m \geq N$, taking $m=n$, and $n>N$, we get, after cancellation:
\[
     \frac{(1-\mathcal{E})}{c(\alpha,\gamma)} +  \frac{(1-\mathcal{E})}{c(\gamma,\beta)} \leq  \frac{(1+\mathcal{E}) 2^{1/d}}{c(\alpha,\beta)}
\]
which upon letting first $\beta$ then $\gamma$ and $\alpha$ to go to infinity, respectively, gives $2 \leq 2^{1/d}$. Hence for $d>1$, $c$ must be $\infty$.

%new material line

For $d=1$, we will use the exhaustion given in \cite[p.~145]{AS} and an associated fundamental hilbertian norm system as explained in the beginning of the section. In this context we will use the \emph{modulus inequality} as given in \cite[p.~14]{S-N}, which in our notation states that for $1\leq p<q<s<\infty$, 
\begin{align} \label{eqn:3.3}
    \frac{1}{c(p,s)} \geq \frac{1}{c(p,q)} + \frac{1}{c(q,s)}
\end{align}

First letting $s\rightarrow \infty$ and then $q \rightarrow \infty$ we see that $1/(c(q))$ approaches to $0$. Hence in this case we conclude that $c=\infty$, as well. So the first possibility does not occur and we conclude that 
\[
    \lim_p \frac{1}{c(p)} = \inf_p \sup_{q \geq p} \frac{1}{c(p,q)} = 0.
\]
The theorem now follows from Corollary \ref{corollary:3.2}.
\end{enumerate}
\end{proof} %of theorem

Now we turn our attention to tameness in the spaces of analytic functions. Recall that a Stein manifold $M$ is called \emph{hyperconvex} in case it possesses a bounded plurisubharmonic exhaustion function. We refer the reader to \cite{KLIMEK} and the references therein for an account of hyperconvex manifolds. From a functional analysis point of view, hyperconvex Stein manifolds $M$ are precisely those Stein manifolds that satisfy, $\mathcal{O}(M) \approx \mathcal{O}(\Delta^d)$, $d=\text{dim}M$, (\cite{ZahH,AytM}). Hence for hyperconvex manifold $M$, $\mathcal{O}(M)$ is a tame Fréchet space. Our next and final result tells us that they are the only ones having this property:

\begin{mainthm} \label{theorem:3.3}
    Let $M$ be a Stein manifold. Then $\mathcal{O}(M)$ is tame if and only if $M$ is hyperconvex.
\end{mainthm}

\begin{proof}
Let us fix a Stein manifold $M$ with dimension $d$. In view of the remarks preceding the theorem, it suffices to show that if $\mathcal{O}(M)$ is tame then $\mathcal{O}(M) \approx \mathcal{O}(\Delta^d)$. In view of Theorem \ref{theorem:3} either $\Delta(\mathcal{O}(M)) = \Delta(\mathcal{O}(\mathbb{C}^d))$ or $\Delta(\mathcal{O}(M)) = \Delta(\mathcal{O}(\Delta^d))$. The first case cannot occur. To see this first observe that the assumption on the diametral dimension implies that $\mathcal{O}(M)$ contains a complemented copy of $\mathcal{O}(\mathbb{C}^d)$ by Theorem 1.3 of \cite{AKT1}. Since plainly tameness passes to complemented subspaces and $\mathcal{O}(\mathbb{C}^d)=\Lambda_\infty(n^{1/d})$ is not tame \cite{DF}, indeed this case cannot occur. Now the theorem follows from Corollary \ref{corollary:3.2} and Theorem \ref{theorem:2}.
\end{proof}

%comment out this if you switch back to 'amsplain' instead of 'plain'
\makeatletter
\renewcommand\@biblabel[1]{#1.}
\makeatother

%\bibliographystyle{plain}
%\bibliography{journal}

%\begin{thebibliography}{HD}

%% Use the widest label as the parameter.
%% Reference items can be numbered or have labels of your choice, as below.

%% In IMPAN journals, only the title is italicized; boldface is not used.
%% Our software will add links to many articles; for this, enclosing volume numbers in { } is helpful
%% Do not give the issue number unless the issues are paginated separately.

%\bibitem[G]{Gratzer} G. Gr\"atzer,
%\emph{More Math into \LaTeX},
%4th ed., Springer, Berlin, 2007.

%\bibitem[HD]{HillDow}  R. Hill and A. Dow,
%\emph{A ground-breaking achievement},
%J.~Differential Equations {15} (1982), 197--211.

%\bibitem[K]{Kow}  J. Kowalski,
%\emph{A very interesting paper},  
%in: Algebra, Analysis and Beyond (Nowhere, 1973),   
%E.~Fox et al. (eds.),
%Lecture Notes in Math. 867, 
%Springer, Berlin, 1974, 115--124.

%\bibitem[N]{Nov} A. S. Novikov,
%\emph{Another fascinating article},  
%Uspekhi Mat. Nauk {23} (1980), no.~3, 112--134 (in Russian); 
%English transl.: Russian Math. Surveys 23 (1980), 572--595.

%\bibitem[R]{Russ} B. Russell,
%\emph{A new theorem},
%arXiv:2012.9876.

%\end{thebibliography}

\end{document}